\documentclass{amsart}

\usepackage[T1]{fontenc}
\usepackage{lmodern}
\usepackage{amsmath,amssymb,amsthm,mathtools}
\usepackage{enumitem}
\usepackage[top=1.6in,bottom=1.6in,left=1.2in,right=1.2in]{geometry}
\usepackage{graphicx,float}
\usepackage{needspace}
\usepackage{tikz}
\usetikzlibrary{arrows.meta}
\usepackage{hyperref}

\newcommand{\NN}{\mathbb{N}}

\newcommand{\law}{\stackrel{\mathrm d}=}

\DeclareMathOperator{\Sym}{Sym}

\setlist[enumerate,1]{label=(\roman*),ref=(\roman*)}

\theoremstyle{plain}
\newtheorem{theorem}{Theorem}[section]

\newtheorem{proposition}[theorem]{Proposition}

\theoremstyle{definition}
\newtheorem{remark}[theorem]{Remark}

\title{Remarks on DAG exchangeability and generalized wreath products}
\author{Paul Jung}
\address{Department of Mathematics, Fordham University,
	407 John Mulcahy Hall, 441 E. Fordham Road,
	Bronx, NY 10458-5165}
\email{paul.jung@fordham.edu}
\date{}
\begin{document}
\begin{abstract}
Motivated by the structure of various standard statistical arrays, Bailey, Praeger, Rowley and Speed (1983) described the automorphism groups of poset block structures as generalized wreath products. We observe that these are exactly the groups used to define DAG exchangeability in Jung, Lee, Staton and Yang (2021), and we restate the representation theorem of the latter work in this language. As a concrete example, we write out the resulting representation for block matrices whose block rows and block columns are separately exchangeable, and whose rows and columns may also be permuted independently inside each block.
\end{abstract}
\maketitle

\section{Introduction}
An infinite array is exchangeable if its joint law is invariant under a prescribed relabeling of its indices. In the familiar case of a matrix $(X_{ij})_{i,j\in\NN}$, separate exchangeability allows us to permute the rows and columns independently. Hoover \cite{Hoover} and Aldous \cite{Aldous81}, independently and through completely different proof methods, gave representations for arrays with such symmetries, and Aldous's Saint-Flour lectures \cite{Aldous85} place these results within the broader theory of exchangeability. The development of graph limits brought renewed attention to exchangeable arrays: sampling from a graphon gives a jointly exchangeable random graph, and every jointly exchangeable simple-graph law is a mixture of such laws \cite{DiaconisJanson,AustinSurvey,LovaszBook}. Later, Austin and Panchenko \cite{AustinPanchenko} considered arrays indexed by paths in rooted trees, where the relabeling at one level is allowed to depend on the labels above it. For each tree, this is the action of an iterated wreath product of symmetric groups.

The work of \cite{JungEtAl} extended this hierarchical picture from trees to finite directed acyclic graphs (DAGs). After its publication, we noticed that its group of allowable relabelings is precisely the generalized wreath product of symmetric groups introduced by Bailey, Praeger, Rowley and Speed \cite{BPRS} for poset block structures. The main purpose of this short note is to make this identification precise (Proposition~\ref{prop:group}) and to restate the representation theorem of \cite{JungEtAl} in the resulting language (Theorem~\ref{thm:rep}).

The concrete question which motivated this note is the following: what is the general form of the law of an exchangeable block matrix, i.e., one in which the block rows and block columns may be permuted separately, as may the rows and columns \emph{inside each block}, with the latter permutations varying from block to block? The representation theorem of \cite{JungEtAl} gives a direct answer, which we write out in \eqref{eq:blockrep} below.

The rest of the note is organized as follows. In Section~\ref{sec:gwp} we describe the generalized wreath product of symmetric groups and show that its action agrees with the DAG automorphisms of \cite{JungEtAl}. In Section~\ref{sec:block} we consider the block matrix example, and in Section~\ref{sec:rep} we state the representation theorem in this notation. Finally, in Section~\ref{sec:stat} we explain how the result relates to the earlier statistical literature and to the theorem of Crane and Towsner \cite{CraneTowsner}.

\section{The generalized wreath product}\label{sec:gwp}
Let $(V,\leq)$ be a finite partially ordered set, and write $\mathcal I=\NN^V$. For $v\in V$, let $\mathcal{A}(v)=\{u\in V:u<v\}$ be its set of strict ancestors. Note that $\mathcal{A}(v)$ is downward closed, since $w<u<v$ implies $w<v$. A tuple $a=(a_u)_{u\in\mathcal{A}(v)}\in\NN^{\mathcal{A}(v)}$ specifies the ancestors' labels; it is the empty tuple when $v$ has no ancestors. For each $v$ and each ancestor tuple $a$, choose a permutation $\sigma_{v,a}\in\Sym(\NN)$ of the possible values of the $v$-coordinate. A fixed family $\sigma=(\sigma_{v,a})$ defines a map $\tau_\sigma$ of $\mathcal I$ in which, at $\alpha\in\mathcal I$, the permutation indexed by $\alpha|_{\mathcal{A}(v)}=(\alpha(u))_{u\in\mathcal{A}(v)}$ acts on $\alpha(v)$:
\begin{equation}\label{eq:action}
 (\tau_\sigma\alpha)(v)=\sigma_{v,\alpha|_{\mathcal{A}(v)}}\bigl(\alpha(v)\bigr),
 \qquad \alpha\in\NN^V.
\end{equation}
Thus, the permutation at $v$ may depend on all of its ancestor labels, but not on any incomparable or descendant label. We denote the set of all maps of the form \eqref{eq:action} by $W(V)$.

Let us check that $W(V)$ is a group of bijections of $\mathcal I$. It contains the identity, obtained by taking every $\sigma_{v,a}$ to be the identity. For two families $\sigma$ and $\sigma'$ we have that
\[
 (\tau_\sigma\tau_{\sigma'}\alpha)(v)
 =\sigma_{v,(\tau_{\sigma'}\alpha)|_{\mathcal{A}(v)}}\Bigl(\sigma'_{v,\alpha|_{\mathcal{A}(v)}}\bigl(\alpha(v)\bigr)\Bigr),
\]
and $(\tau_{\sigma'}\alpha)|_{\mathcal{A}(v)}$ depends only on $\alpha|_{\mathcal{A}(v)}$ since $\mathcal{A}(v)$ is downward closed. Thus $\tau_\sigma\tau_{\sigma'}$ again has the form \eqref{eq:action}. It remains to show that $\tau_\sigma$ is invertible with an inverse of the same form. To see this, fix a \emph{linear extension} $v_1,\ldots,v_m$ of $V$, i.e., a listing in which $u<v$ implies that $u$ is placed before $v$. (Two incomparable ancestors of a vertex may occur in either order, but both precede that vertex.) Given $\beta=\tau_\sigma\alpha$, we recover $\alpha(v_1),\ldots,\alpha(v_m)$ in this order. At step $k=1,\ldots,m$, every ancestor of $v_k$ occurred at an earlier step, so its original label is already known and
\[
 (\tau_\sigma^{-1}\beta)(v_k)=\alpha(v_k)
 =\sigma_{v_k,\alpha|_{\mathcal{A}(v_k)}}^{-1}\bigl(\beta(v_k)\bigr).
\]
By induction on $k$, and again since $\mathcal{A}(v_k)$ is downward closed, the recovered tuple $\alpha|_{\mathcal{A}(v_k)}$ depends only on $\beta|_{\mathcal{A}(v_k)}$. Therefore $\tau_\sigma^{-1}$ also has the form \eqref{eq:action}.

Recall that a \emph{chain} is a poset in which every two vertices are comparable. For a two-vertex chain $u<v$, with $\alpha(u)=i$ and $\alpha(v)=j$, the action \eqref{eq:action} reads $(i,j)\mapsto(\pi(i),\sigma_i(j))$. This is the basic \emph{permutation wreath product}: one permutation acts on the outer label $i$, while the permutation of the inner label $j$ may vary with $i$. Longer chains iterate this action, while incomparable vertices allow separate coordinate permutations.

Bailey, Praeger, Rowley and Speed illustrate the difference between a matrix which is exchangeable with respect to a direct product of symmetric groups and one which is exchangeable with respect to a wreath product of symmetric groups in their Figures~1 and~2, reproduced below in Figure~\ref{fig:bprs-actions}. Let $(G_1,\Delta_1)$ act on row labels $i$ and $(G_2,\Delta_2)$ on column labels $j$. In their Figure~1 (the left panel of Figure~\ref{fig:bprs-actions}) the direct product sends $(i,j)$ to $(g_1(i),g_2(j))$, so the same row permutation $g_1\in G_1$ is used in every column. In their Figure~2 (the right panel) the wreath product sends $(i,j)$ to $(f_j(i),g_2(j))$, where $f:\Delta_2\to G_1$ specifies a row permutation separately for each column.

Thus rows are globally identifiable in the first case, whereas the second case has no globally defined rows; in the terminology of \cite{BPRS}, the column index dominates the row index. In the language of experimental design, rows and columns are crossed in the left panel, while rows are nested within columns in the right panel. The latter is the two-vertex chain action above, with its outer (column) coordinate written second in $(i,j)$.
\begin{figure}[H]
\centering
\begin{minipage}[c]{0.48\linewidth}
\centering
\includegraphics[width=\linewidth,height=1.2in,keepaspectratio]{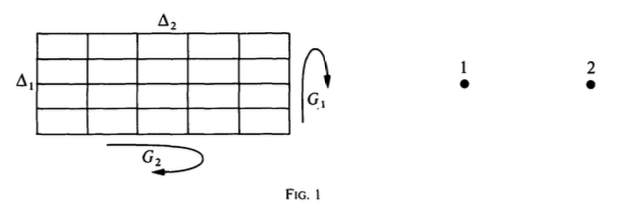}
\end{minipage}\hfill
\begin{minipage}[c]{0.48\linewidth}
\centering
\includegraphics[width=\linewidth,height=1.2in,keepaspectratio]{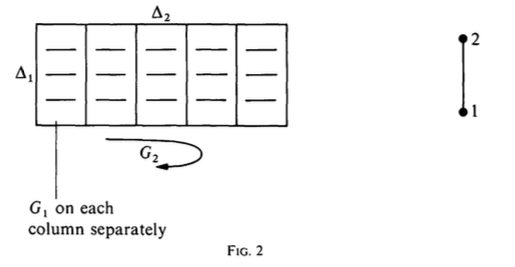}
\end{minipage}
\caption{Figures~1 and~2 of \cite{BPRS} (printed p.~69): the direct product permutes all rows together (left), while the wreath product permits a separate permutation of rows within each column before the columns are permuted (right).}
\label{fig:bprs-actions}
\end{figure}

Bailey, Praeger, Rowley and Speed \cite[Sections 2--3]{BPRS} work with label sets $\Delta_i$ and permutation groups $G_i$ indexed by a poset $(I,\rho)$. With $A(i)$ the set of \emph{dominators} of $i$ and $\Delta^i=\prod_{j\in A(i)}\Delta_j$, each function $f_i:\Delta^i\to G_i$ selects the permutation $f_i(\delta|_{A(i)})$ which acts on $\delta_i$. Their label sets and groups are general; we only need the case where every label set is $\NN$ and every group is the full symmetric group.

\begin{samepage}
Set $I=V$, $\Delta_v=\NN$, $G_v=\Sym(\NN)$, and take $\rho$ opposite to $\leq$, so that their dominators become our ancestors. Then $\Delta^v=\NN^{\mathcal{A}(v)}$ and $f_v(a)=\sigma_{v,a}$, and in their notation (see \cite[Theorem A]{BPRS}) our permutation group on $\mathcal I=\NN^V$ is
\begin{equation}\label{eq:bprsnotation}
 (W(V),\mathcal I)
 =\prod_{(I,\rho)}(G_v,\Delta_v)
 =\prod_{(V,\leq^{\mathrm{op}})}(\Sym(\NN),\NN).
\end{equation}
\end{samepage}
Let $\mathcal D(V)$ denote the set of downward-closed subsets (downsets) of $V$. Recall that a DAG automorphism, which is called a $G$-automorphism in \cite[Definition 1.3]{JungEtAl} (where $G$ denotes the DAG), is a bijection $\tau:\NN^V\to\NN^V$ such that, for every $D\in\mathcal D(V)$ and all $\alpha,\beta\in\NN^V$,
\begin{equation}\label{eq:preserve}
 \alpha|_D=\beta|_D\quad\Longleftrightarrow\quad
 (\tau\alpha)|_D=(\tau\beta)|_D.
\end{equation}
In other words, two indices agree on a downset of coordinates exactly when their images agree there.

\begin{proposition}\label{prop:group}
The group of DAG automorphisms of $\NN^V$ is $W(V)$.
\end{proposition}
\begin{proof}
We check both inclusions. Suppose first that $\tau$ has the form \eqref{eq:action}. If $\alpha$ and $\beta$ agree on a downset $D$, then their transformed coordinates in $D$ also agree, since every ancestor of a vertex in $D$ belongs to $D$. The inverse of $\tau$ has the same form, so the converse implication in \eqref{eq:preserve} holds as well.

Now suppose that $\tau$ satisfies \eqref{eq:preserve}. Fix $v\in V$ and put $D_v=\mathcal{A}(v)\cup\{v\}$, which is a downset. By \eqref{eq:preserve}, the $v$-coordinate of $\tau\alpha$ depends only on $\alpha|_{D_v}$. Moreover, for any downset $D$, the map $\tau$ induces a bijection $\tau_D$ on $\NN^D$ with $(\tau\alpha)|_D=\tau_D(\alpha|_D)$: the forward and reverse implications in \eqref{eq:preserve} make the induced map well defined and injective, and surjectivity follows from that of $\tau$ (this induced map is also noted at the beginning of \cite[Section 3]{JungEtAl}). Hence $\tau_{D_v}$ maps the fiber over each $a\in\NN^{\mathcal{A}(v)}$ bijectively onto the fiber over $\tau_{\mathcal{A}(v)}a$. On the remaining coordinate this fiber map must be a permutation $\sigma_{v,a}$ of $\NN$. Doing this for each $v$ gives \eqref{eq:action}, completing the proof.
\end{proof}

\begin{remark}\label{rem:parents}
Example~2.3(d) of \cite{JungEtAl} describes the permutation at a vertex $v$ as depending only on the labels of the parents of $v$ in the DAG. Proposition~\ref{prop:group} shows that the correct description allows the permutation to depend on the labels of all ancestors of $v$. The two descriptions already differ for a three-vertex chain $u<w<v$: the map $(i,j,k)\mapsto(i,j,\sigma_{i,j}(k))$ satisfies \eqref{eq:preserve} for every choice of the permutations $\sigma_{i,j}$, whereas dependence on the parent $w$ alone would force $\sigma_{i,j}$ not to depend on $i$. Dependence on all ancestors is also what one expects from the rooted-tree automorphisms of \cite{AustinPanchenko}. The results of \cite{JungEtAl} are stated and proved in terms of Definition~1.3 there, and so they are unaffected.
\end{remark}

Thus a random array $X=(X_\alpha)_{\alpha\in\NN^V}$ taking values in a standard Borel space is DAG exchangeable precisely when
\begin{equation}\label{eq:invariance}
 (X_\alpha)_{\alpha\in\NN^V}\ \law\
 (X_{\tau\alpha})_{\alpha\in\NN^V},\qquad \tau\in W(V).
\end{equation}
Here and below, equality in distribution refers to the whole array.

\Needspace{2.7in}
\section{A matrix of exchangeable blocks}\label{sec:block}
Let us now consider the block matrix example. Let $V=\{R,C,r,c\}$, with $R$ and $C$ denoting the row and column of a block and $r$ and $c$ the row and column within it. Put
\[
 R<r,\qquad C<r,\qquad R<c,\qquad C<c,
\]
with no other comparabilities (Figure~\ref{fig:blockdag}).
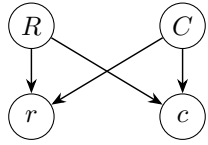
\begin{figure}[H]
\centering
\begin{tikzpicture}[
  >=Stealth,
  vertex/.style={circle,draw,minimum size=6mm,inner sep=0pt},
  edge/.style={->,line width=0.55pt}]
  \node[vertex] (R) at (-1,0.6) {$R$};
  \node[vertex] (C) at (1,0.6) {$C$};
  \node[vertex] (r) at (-1,-0.6) {$r$};
  \node[vertex] (c) at (1,-0.6) {$c$};
  \draw[edge] (R) -- (r);
  \draw[edge] (R) -- (c);
  \draw[edge] (C) -- (r);
  \draw[edge] (C) -- (c);
\end{tikzpicture}
\caption{The block-matrix DAG: each block coordinate $R,C$ precedes both within-block coordinates $r,c$. Redrawn from Figure~2.3 of \cite[Example 2.2(b)]{JungEtAl}.}
\label{fig:blockdag}
\end{figure}
We write $X_{i,j,a,b}$ for entry $(a,b)$ in block $(i,j)$. In this case \eqref{eq:action} reads
\begin{equation}\label{eq:blockaction}
 (i,j,a,b)\longmapsto
 \bigl(\pi(i),\pi'(j),\sigma_{i,j}(a),\kappa_{i,j}(b)\bigr),
\end{equation}
where $\pi,\pi'\in\Sym(\NN)$, and $\sigma_{i,j},\kappa_{i,j}\in\Sym(\NN)$ may be chosen separately for every $(i,j)$ (in the notation of \eqref{eq:action}, $\sigma_{i,j}=\sigma_{r,(i,j)}$ and $\kappa_{i,j}=\sigma_{c,(i,j)}$). In other words, we may permute the block rows and block columns, and also permute the rows and columns inside each block, using a separate pair of permutations for every block. In particular, the inner row label $a$ does not identify a row shared with another block.

It is the last point which explains the two additional edges $C<r$ and $R<c$. If we used only the disjoint chains $R<r$ and $C<c$, a permutation of the inner rows in block row $i$ would have to be used in every block $(i,j)$ of that block row. To see the difference, let $(Z_{i,a})$ be independent, identically distributed, nondegenerate random variables and set $X_{i,j,a,b}=Z_{i,a}$. This array is exchangeable under the two-chain group, but it is not invariant under \eqref{eq:blockaction}: permuting the inner rows in one block destroys the almost-sure agreement with the other blocks in the same block row. On the other hand, using independent $Z_{i,j,a}$ in place of $Z_{i,a}$ gives an array which is invariant under \eqref{eq:blockaction}. This small example shows which dependence the extra edges remove.

There are seven downsets of $V$:
\[
 \varnothing,\ \{R\},\ \{C\},\ \{R,C\},\ \{R,C,r\},\ \{R,C,c\},\ \{R,C,r,c\}.
\]
Theorem~\ref{thm:rep} below therefore gives the following representation, jointly over all entries:
\begin{equation}\label{eq:blockrep}
 (X_{i,j,a,b})_{i,j,a,b\in\NN}\ \law\
 \left(f\bigl(U_\varnothing,U^R_i,U^C_j,U^{RC}_{i,j},
 U^{RCr}_{i,j,a},U^{RCc}_{i,j,b},U^{RCrc}_{i,j,a,b}\bigr)\right)_{i,j,a,b\in\NN}.
\end{equation}
Here all distinct $U$'s are independent uniform random variables. In particular, $U^{RCr}_{i,j,a}$ can produce dependence along inner row $a$ within block $(i,j)$, while no variable indexed by $(i,a)$ alone imposes an alignment of inner rows across different block columns. The seven inputs thus record exactly the levels at which the group permits entries to share randomness.

\section{The representation theorem}\label{sec:rep}
The following is Theorem~1.5 of \cite{JungEtAl}, restated using Proposition~\ref{prop:group}, together with its converse.
\begin{theorem}[Jung--Lee--Staton--Yang \cite{JungEtAl}]\label{thm:rep}
Let $V$ be a finite poset and let $X=(X_\alpha)_{\alpha\in\NN^V}$ take values in a standard Borel space. Then the law of $X$ is invariant under the action on $\NN^V$ of the generalized wreath product of symmetric groups
\[
 W(V)=\prod_{(V,\leq^{\mathrm{op}})}(\Sym(\NN),\NN),
\]
i.e., $(X_\alpha)_{\alpha\in\NN^V}\law(X_{\tau\alpha})_{\alpha\in\NN^V}$ for every $\tau\in W(V)$, if and only if there exist a measurable function $f$ and independent uniform $[0,1]$ random variables $\{U_{D,a}:D\in\mathcal D(V),\ a\in\NN^D\}$ such that
\begin{equation}\label{eq:representation}
 (X_\alpha)_{\alpha\in\NN^V}\ \law\
 \left(f\bigl((U_{D,\alpha|_D})_{D\in\mathcal D(V)}\bigr)\right)_{\alpha\in\NN^V}.
\end{equation}
\end{theorem}
\begin{proof}
By Proposition~\ref{prop:group}, $W(V)$ is precisely the group of automorphisms used in \cite[Theorem 1.5]{JungEtAl}. That theorem gives the forward implication, with one independent uniform variable for every restricted tuple $a\in\NN^D$, including $U_{\varnothing,\varnothing}$, which is common to all entries.

The converse is also noted in \cite[Section 3]{JungEtAl}; we include the short argument for completeness. Fix $\tau\in W(V)$. By the proof of Proposition~\ref{prop:group}, $\tau$ induces a bijection $\tau_D:\NN^D\to\NN^D$ for each downset $D$, with $(\tau\alpha)|_D=\tau_D(\alpha|_D)$. Since $\{U_{D,\tau_D(a)}\}_{D,a}$ has the same joint law as $\{U_{D,a}\}_{D,a}$, applying $f$ at every $\alpha$ gives \eqref{eq:invariance}, completing the proof.
\end{proof}

\begin{remark}
The argument above proves the reformulation, using the original representation theorem for its forward implication. Theorem~3.2 of \cite{JungEtAl} also gives a simultaneous representation for a collection of arrays indexed by different downsets; Proposition~\ref{prop:group} translates its invariance assumption in the same way.
\end{remark}

\section{Relation to the statistical literature}\label{sec:stat}
Nelder \cite{Nelder} studied crossed and nested factors in randomized experiments. Bailey, Praeger, Rowley and Speed \cite[Theorem B]{BPRS} identified the automorphism group of a poset block structure as a generalized wreath product, and Bailey \cite{Bailey1991,BaileyBook} used such permutation groups to study randomization and strata. Thus the symmetry in the block matrix example was already a natural one in statistics. As far as we know, however, those references did not consider a representation of all laws of infinite arrays which are invariant under this action.

On the level of representations, a theorem of Crane and Towsner \cite{CraneTowsner} implies the single-array representation once the DAG is encoded by equivalence relations; Appendix~A of \cite{JungEtAl} carries out this translation. Their general theorem indexes randomness by antichains of equivalence classes, called ``blurs,'' whereas \eqref{eq:blockrep} displays the seven inputs for block matrices directly. In particular, the latter is an explicit representation in the natural coordinates of the data structure, while the former provides the first general theorem from which such a representation can be deduced. We do not know whether the simultaneous representation of \cite[Theorem 3.2]{JungEtAl} can be deduced in the same way; this was left open in \cite{JungEtAl}.

\end{document}